\documentclass{article}

 \PassOptionsToPackage{numbers, compress}{natbib}

    \usepackage[preprint]{neurips_2019}
    
    \usepackage{amsmath,amsthm,bm}
    \usepackage{graphicx,graphics,subcaption}
    \usepackage{xcolor}
    \usepackage{hyperref}
    \usepackage{multicol}
    \newcommand{\norm}[1]{\left\lVert#1\right\rVert}
    
    \newcommand{\report}[1]{}
    
    \newcommand{\rev}[1]{#1}
    \newcommand{\revtwo}[1]{#1}
    \newtheorem{theorem}{Theorem}[section]
    
    \newtheorem{lemma}[theorem]{Lemma}

    \usepackage{graphics,graphicx,color,epsfig}
\usepackage{tikz}
\usetikzlibrary{automata, positioning, arrows,shapes.multipart,fit}
\usepackage{pgfplots}

\usepackage[utf8]{inputenc} 
\usepackage[T1]{fontenc}    
\usepackage{hyperref}       
\usepackage{url}            
\usepackage{booktabs}       
\usepackage{amsfonts}       
\usepackage{nicefrac}       
\usepackage{microtype}      

\title{Concavifiability and convergence: necessary and sufficient conditions for gradient descent analysis}

\author{%
  Thulasi Tholeti \\
  Department of Electrical Engineering\\
  IIT Madras\\
  India, 600036 \\
  \texttt{ee15d410@ee.iitm.ac.in} \\
   \And
  Sheetal Kalyani \\
  Department of Electrical Engineering\\
  IIT Madras\\
  India, 600036 \\
  \texttt{skalyani@ee.iitm.ac.in} \\
}

\begin{document}
\maketitle

\begin{abstract}
    Convergence of the gradient descent algorithm has been attracting renewed interest due to its utility in deep learning applications. Even as multiple variants of gradient descent were proposed, the assumption that the gradient of the objective is Lipschitz continuous remained an integral part of the analysis until recently. In this work, we look at convergence analysis by focusing on a property that we term as concavifiability, instead of Lipschitz continuity of gradients. We show that concavifiability is a necessary and sufficient condition to satisfy the upper quadratic approximation which is key in proving that the objective function decreases after every gradient descent update. We also show that any gradient Lipschitz function satisfies concavifiability. A constant known as the concavifier analogous to the gradient Lipschitz constant is derived which is indicative of the optimal step size. As an application, we demonstrate the utility of finding the concavifier the in convergence of gradient descent through an example inspired by neural networks. We derive bounds on the concavifier to obtain a fixed step size for a single hidden layer ReLU network.
\end{abstract}

\section{Introduction}
Gradient descent is a well-known iterative optimization algorithm employed to minimize a function. It operates by taking steps in the direction opposite to that of the direction of the gradient at that point. Being a first-order method that uses gradients, it is less computationally intensive than other optimization methods which employ second-order Hessian information \cite{battiti1992first}. This attribute is particularly useful while dealing with high dimensional problems \cite{bottou2010large}. Gradient descent and its stochastic variants have now regained attention as these methods are avidly employed in training deep neural networks.

Gradient descent and its stochastic version are well-studied for convex functions where convergence to the unique minimum of the function was guaranteed \cite{nesterov1998introductory}. Over and above convexity, one of the basic assumptions made is that the gradients of the objective obey Lipschitz continuity (see Definition 1).
\paragraph{Definition 1: Gradient Lipschitz function} A differentiable function \(f: \mathbb{R}^d \rightarrow \mathbb{R}\) is said to be \(L\)- gradient Lipschitz if for any \(\bm{x_1},\bm{x_2}\) in the domain of \(f\), and for \(L > 0\),
\begin{equation} \label{eqn:lipschitz}
    \norm{ \nabla f(\bm{x_1}) - \nabla f(\bm{x_2}) } \leq L  \norm{\bm{x_1} - \bm{x_2} }
\end{equation}

The Lipschitz gradient assumption allows the function to have an upper quadratic approximation which is essential to derive a descent lemma (Lemma \ref{lemma:descent}) that shows that the objective value decreases after every iteration of gradient descent. 

Recently, there has been interest in extending these iterative optimization algorithms to a more general setting. In case of non-convex functions, gradient descent is guaranteed to converge to a first-order stationary point \cite{jin2017escape}. There has been focus on the ability of the gradient descent algorithm to escape saddle points for a non-convex function in \cite{jin2017escape,Lee2019,du2017gradient} and accelerating the convergence in \cite{nitanda2014stochastic,johnson2013accelerating}, where the objective function is assumed to be gradient Lipschitz.

With the advent of machine learning came the necessity to optimize non-smooth objective functions due to attributes like hinge loss in Support Vector Machines and Rectified Linear Unit (ReLU) activation functions in neural networks. In \cite{shamir2013stochastic}, the authors provide finite-sample bounds on the optimization error of individual iterates for stochastic gradient descent applied on a non-smooth but convex function. Addressing the problem of non-smooth as well as non-convex optimization, an alternate descent lemma was derived in \cite{bolte2018first} for composite objectives, i.e., the objective is of the form \(f +g\). The authors introduce a class of functions known as smooth adaptable functions for which the alternate descent lemma applies. However, this was in context of the proximal descent algorithm where finding every iterate is a minimization problem on a compact set. The work in \cite{khamaru2018convergence} also discusses non-smooth and non-convex optimization for composite functions. The descent lemma is derived with the additional assumption that a component of the composite objective is smooth. In the literature that deals with the analysis of gradient descent for composite functions, although the sum of the components is non-smooth, conditions are imposed on the individual components \cite{bolte2018first,khamaru2018convergence}. Moreover, given a function, the authors do not specify if it can be brought to the composite form as required by the analysis. 

\paragraph{Contributions} We initially discuss the classical descent lemma and the conditions that are typically used to derive the same. We then extend the applicability of the descent lemma to a class of functions that we define as concavifiable functions characterized by a constant that we term as concavifier. We then show that concavifiability is a necessary and sufficient condition to derive the upper quadratic approximation which is instrumental in proving the traditional descent lemma. \revtwo{When gradient descent is employed in training neural networks, the step size needs to be chosen carefully and is one of the important hyper-parameters to be optimized for before training the network. In our work, we consider the case of the optimization problem in training a single hidden network with ReLU activation function. We derive an expression for explicitly computing the concavifier. As the computation will involve a brute force search, we provide upper bounds on the optimal concavifier and we also demonstrate the performance of gradient descent when step sizes are chosen according to the derived bounds. The bounds derived in our work can aid in finding the optimal step size for training thereby avoiding a grid search.}

\paragraph{Notation}
We consider a continuous function \(f(\bm{x}): \mathbb{R}^d \rightarrow \mathbb{R}\) where \(\bm{x} \in \mathbb{R}^d\). We denote bold upper-case letters \(\bm{A},\bm{B}\) to denote matrices and \(a_{ij},b_{ij}\) to denote their \((i,j)\)th elements respectively. The maximum eigen value of \(\bm{A}\) is denoted as \(\lambda_{max}(\bm{A})\). The bold lower-case letters \(\bm{x},\bm{y}\) denote vectors.  All vectors are column vectors unless stated otherwise. The \(\ell_2 \) norm of a vector is denoted as \(\norm{.}\).  The indicator function denoted as \(\mathbb{I}_{\mathcal{E}}\) takes the value 1 when \(\mathcal{E}\) is true and value 0 otherwise.

\section{Classical gradient descent analysis}
We consider the classic problem of unconstrained minimization of a function \(f(x)\). Gradient descent is a well-established algorithm to solve this problem iteratively. The algorithm picks an initial point and iteratively moves towards the minimum by taking a step in the direction opposite to the gradient at that point. The update at any step \(t\) is given by 
\begin{equation} \label{eqn:updateEqn}
    \bm{x_t} = \bm{x_{t-1}} - \eta \nabla f(\bm{x_{t-1}})
\end{equation}
where \(\eta\) refers to the fixed step size.
The analysis of this algorithm is now being re-looked at, primarily because of the use of iterative optimization algorithms in training neural networks. Previously, in the analysis of gradient descent(e.g., \cite{nesterov1998introductory}) as well as its variants, the basic assumption made is that the function to be minimized is gradient Lipschitz continuous, i.e., it satisfies Equation \ref{eqn:lipschitz}. \rev{This property is required so that an upper quadratic approximation (Equation \ref{eqn:quadApprox}) can be given for the function. The term upper quadratic approximation is defined below.}

\paragraph{Definition 2: Upper quadratic approximation} The upper quadratic approximation of a differentiable function \(f: \mathbb{R}^d \rightarrow \mathbb{R}\) is said to hold with a constant \(c \geq 0\) if for all \(\bm{x},\bm{y} \in \mathbb{R}^d\),
\begin{equation} \label{eqn:quadApprox}
        f(\bm{y}) \leq f(\bm{x}) + \nabla f(\bm{x})^T(\bm{y}-\bm{x}) + \dfrac{c}{2} \norm{\bm{y}-\bm{x}}^2.
    \end{equation}
    
\rev{We use Lemma \ref{lemma: quadApprox} from \cite{nesterov1998introductory} to show that gradient Lipschitz functions satisfy the upper quadratic approximation; this property is, in turn, used to prove that the objective function decreases after every iteration of the algorithm in Lemma \ref{lemma:descent}.}
\begin{lemma} \label{lemma: quadApprox}
    If a function \(f\) is \(L\)-gradient Lipschitz , the upper quadratic approximation with a constant \(L\) holds.
\end{lemma}

\begin{lemma}\textbf{(Descent lemma) \cite{nesterov1998introductory}} \label{lemma:descent}
    For an \(L\)-gradient Lipschitz function \(f: \mathbb{R}^d \rightarrow \mathbb{R}\), gradient descent with a step size \(\eta \leq 1/L\) produces a decreasing sequence of objective values and the optimal step size is given by \(\eta^* = 1/L\).
\end{lemma}
\begin{proof}  
    For an \(L\)-gradient Lipschitz function, using Lemma \ref{lemma: quadApprox}, 
    \begin{equation}
        f(\bm{y}) \leq f(\bm{x}) + \nabla f(\bm{x})^T(\bm{y}-\bm{x}) + \dfrac{L}{2} \norm{\bm{y}-\bm{x}}^2 \forall \bm{x},\bm{y} \in \mathbb{R}^d
    \end{equation}
    Setting $\bm{x}= \bm{x_t}$ and $\bm{y} = \bm{x_{t+1}} = \bm{x_t} - \eta \nabla f(\bm{x_t})$,
    \begin{align}
    f(\bm{x_{t+1}}) &\leq f(\bm{x_t}) + \nabla f(\bm{x_t})^T(\bm{x_t} - \eta \nabla f(\bm{x_t})-\bm{x_t}) + 
    \dfrac{L}{2} \norm{\bm{x_t} - \eta \nabla f(\bm{x_t})-\bm{x_t}}^2 \nonumber\\
    f(\bm{x_{t+1}}) &\leq f(\bm{x_t}) - \left(1 - \dfrac{L \eta}{2}\right) \eta \norm{\nabla f(\bm{x_t})}^2 \nonumber \\
    f(\bm{x_{t+1}}) &\leq f(\bm{x_t}) - \dfrac{\eta}{2}  \norm{\nabla f(\bm{x_t})}^2 \label{eqn:dec}\\
    f(\bm{x_{t+1}}) &\leq f(\bm{x_t})
\end{align}
We use the condition that the step size \( \eta \leq 1/L\) to obtain Equation \ref{eqn:dec} and show that \(f(\bm{x_t})\) decreases with increase in \(t\) until convergence as the quantity \(\norm{\nabla f(\bm{x_t})}^2\) is non-negative. To achieve the maximum possible decrease in the functional value after a gradient descent iteration, we need to maximize the quantity \(f(\bm{x_t}) - f(\bm{x_{t+1}}) \) which is equivalent to maximizing \(\left(1 - \dfrac{L \eta}{2}\right) \eta\) with respect to the step size \(\eta\). Therefore, \(\eta =1 /L\) is the optimal choice for the step size.
\end{proof}

Note that in case the objective function is convex, Lemma \ref{lemma:descent} proves that gradient descent converges to the global minimum. For a non-convex problem, Lemma \ref{lemma:descent} shows convergence to a first-order stationary point; i.e., a point where the gradient of the function is zero.

\section{Can the descent lemma be proved for a broader class of functions?}
The descent lemma is essential to establish the convergence of gradient descent algorithm. We make the following important observation: the upper quadratic approximation presented in Equation \ref{eqn:quadApprox} is instrumental in proving the descent lemma and not the gradient Lipschitz property (Equation \ref{eqn:lipschitz}) itself. In this section, we ask if there exists a class of functions that is necessary and sufficient to achieve the upper quadratic approximation. In this context, the work in \cite{bolte2018first} defines a class of functions known as smooth adaptable functions as follows: a pair \((g,h)\) is \(L\)-smooth adaptable if \(Lh-g\) is convex for a convex \(h\). Independently, a class of functions known as the convexifiable functions which are of the form \(f(\bm{x}) - \frac{\alpha}{2}\norm{\bm{x}}^2\) for a constant \(\alpha\), was introduced in \cite{zlobec2006characterization}; \rev{the major focus was to find convexifiable functions so that those class of problems can also be optimized using linear and convex programming.} Drawing inspiration from the above two definitions, we define the following class of functions.

\paragraph{Definition 3: Concavifiable functions} A function \(f\) is defined to be concavifiable if the function \[g(\bm{x}) = \dfrac{\alpha}{2} \norm{\bm{x}}^2 - f(\bm{x})\] is convex, where the constant \(\alpha>0\) is termed as the concavifier. The least value of \(\alpha\) for which \(g(x)\) is convex is termed as the optimal concavifier, denoted as \(\alpha^*\).\\
Note that any \(\alpha>\alpha^*\) will also be a concavifier for \(f(\bm{x})\). In \cite{zlobec2006characterization}, the class of functions which are a quadratic term away from a convex function are termed as convexifiable. Similarly, according to our definition, the function \(f(\bm{x})\) is a quadratic term away from a concave function \(- g(\bm{x})\); therefore, we term these class of functions as concavifiable functions. The scale of the quadratic term is given by the concavifier. We now argue that concavifiability is a necessary and sufficient condition for obtaining Equation \ref{eqn:quadApprox}. 
\begin{lemma} \label{lemma:concavifiable}
    A differentiable function \(f\) is concavifiable with a concavifier \(\alpha\) if and only if it satisfies the upper quadratic approximation with the constant \(\alpha \geq 0\).
\end{lemma}
\begin{proof}
    Let us assume that the function \(f\) is concavifiable. For a concavifiable function \(f\), \(g(\bm{x}) = \dfrac{\alpha}{2} \norm{\bm{x}}^2 - f(\bm{x})\) is convex. Therefore, by the first-order condition for convexity, the function \(g\) is convex if and only if 
    \begin{align}
        g(\bm{y}) &\geq g(\bm{x}) + \nabla g(\bm{x})^T (\bm{y}-\bm{x}).  \\
        \dfrac{\alpha}{2}\norm{\bm{y}}^2 - f(\bm{y}) &\geq \dfrac{\alpha}{2}\norm{\bm{x}}^2 - f(\bm{x}) + (\alpha \bm{x} - \nabla f(\bm{x}) )^T(\bm{y}-\bm{x}) \\
        f(\bm{y}) &\leq f(\bm{x}) + \dfrac{\alpha}{2} \left[ \norm{\bm{y}}^2 - \norm{\bm{x}}^2 \right] - (\alpha \bm{x} - \nabla f(\bm{x}))^T (\bm{y}-\bm{x}) \\
        f(\bm{y}) &\leq f(\bm{x}) + \nabla f(\bm{x})^T(\bm{y}-\bm{x}) + \dfrac{\alpha}{2} \norm{\bm{y}-\bm{x}}^2.
    \end{align}
    This shows that \(f\) satisfies the upper quadratic equation with constant \(\alpha\).
    \end{proof}

Lemma \ref{lemma:concavifiable} not only shows that concavifiable functions satisfy the upper quadratic approximation. It also shows that this class cannot be extended any further as concavifiability is a necessary as well as a sufficient condition for the upper quadratic approximation. \rev{In \cite{zhou2018fenchel}, the author illustrates the equivalence of the upper quadratic approximation with constant \(L\) and functions of the form \(L \norm{\bm{x}}^2 /2 - f(\bm{x})\) which is a special case of the result in Lemma \ref{lemma:concavifiable} when the concavifier is the Lipschitz constant \(L\). The next theorem shows that, for the class of concavifiable functions, gradient descent produces a decreasing sequence of objective values.} 

\begin{theorem} \label{theorem: main}
    Assume that \(f\) is a concavifiable function with a concavifier \(\alpha>0\). Then, on applying gradient descent with a step size \(\eta \leq 1/\alpha \), for any time step \(t\), we get,
    \begin{equation} \label{eqn:des}
        f(\bm{x_{t+1}}) \leq f(\bm{x_t}) - \dfrac{\eta}{2}  \norm{\nabla f(\bm{x_t})}^2.
    \end{equation}
    \rev{The optimal step size is given by \(\eta^* = 1/\alpha^*\) where \(\alpha^*\) is the lowest possible concavifier.}
\end{theorem}
\begin{proof}
    The proof is obtained by using Lemma \ref{lemma:descent} in the context of Lemma \ref{lemma:concavifiable}. \rev{Lemma \ref{lemma:descent} shows that a function for which the upper quadratic approximation holds with a constant \(L\), the optimal step size is given by \(\eta^* = 1/L\). By a similar argument, the optimal step size for a concavifiable function (for which the upper quadratic approximation holds with constant \(\alpha\)) is \(\eta = 1/\alpha\). However, we note that to achieve the maximum descent (i.e. \(f(\bm{x_t}) -f(\bm{x_{t+1}})\)), we need to maximize the step size \(\eta\) which is achieved for the optimal concavifier. As \(\eta\) is maximized for the lowest value of \(\alpha\), the optimal step size will be \(\eta^* = 1/\alpha^*\).}
\end{proof}

\subsection{Importance of the concavifier}
\rev{In this section, we bring out the significance of finding the concavifier. As pointed out by Theorem \ref{theorem: main}, the inverse of the optimal concavifier gives the best fixed value step size that one can employ for gradient descent. For training neural networks, we need to select hyper-parameters such as learning rate, number of hidden layers, etc. which was typically done through a grid search as suggested in \cite{bergstra2012random} or more recently, through Bayesian optimization \cite{snoek2012practical}. These methods take up a significant amount of compute resources. In this work, by explicitly computing the optimal step size, we do away with the necessity to tune the learning rate parameter. Even for the case of more complicated networks where an exact computation is not possible, we can use the upper bound on the concavifier to obtain a restricted range over which a search can be performed with lesser computational effort in order to tune the learning rate.}

\subsection{Finding the concavifier}
In this section, we propose a method to characterize the value of the optimal concavifier for a given concavifiable function \(f\), i.e., the lowest value of the concavifier for which the function is concavifiable. We are interested in finding the optimal concavifier in the context of gradient descent as the inverse of the optimal step size is indicative of the optimal step size as pointed out in Theorem \ref{theorem: main}. \\
We now discuss how to find the concavifier for the class of doubly differentiable functions. The method employed is similar to the work in \cite{zlobec2003estimating}.

\begin{lemma} \label{lemma:doubleDiff}
    All doubly differentiable functions are concavifiable and the optimal concavifier is given by \[\alpha^* = \max_{\bm{x}} \lambda_{max}(\nabla^2 f(\bm{x})).\]
\end{lemma}
\begin{proof}
    For a doubly differentiable function \(f(\bm{x})\), consider the function \(g(\bm{x}) =  \dfrac{\alpha}{2} \norm{\bm{x}}^2 - f(\bm{x})\). The second derivative is given by \(\nabla^2 g(\bm{x}) = \alpha \bm{I}_d- \nabla^2 f(\bm{x})\) with eigen values \(\alpha - \lambda_i(\nabla^2 f(\bm{x}))\) for \(i=1,\cdots,d\). Note that the Hessian of a convex function should be positive semi-definite (PSD). Therefore, we pick \(\alpha = \alpha^* = \max_{\bm{x}} \lambda_{max}(\nabla^2 f(\bm{x}))\) to ensure that all the eigen values of \(\nabla^2g(\bm{x})\) are non-negative. Hence, we can conclude that \(g(\bm{x})\) is convex and that \(f(\bm{x})\) is concavifiable. Note that any $\alpha > \alpha^*$ also functions as a concavifier.
\end{proof}

\begin{lemma}
    All \(L\)-gradient Lipschitz functions are concavifiable with optimal concavifier \(\alpha^* = L\).
\end{lemma}
\begin{proof}
    For Lipschitz functions, by Cauchy Schwarz inequality,
    \begin{align*}
        \left| \langle \nabla f(\bm{x}) - \nabla f(\bm{y}), \bm{x} - \bm{y} \rangle \right| &\leq \norm{ \nabla f(\bm{x}) - \nabla f(\bm{y})} \norm{\bm{x}-\bm{y}} \\
        &\leq L\norm{\bm{x}-\bm{y}}^2
    \end{align*}
    For \(g(\bm{x}) = \frac{\alpha}{2} \norm{\bm{x}}^2 - f(\bm{x})\) with \(\alpha >L\),
    \begin{align*}
        \langle \nabla f(\bm{x}) - \nabla f(\bm{y}), \bm{x} - \bm{y} \rangle &\leq \alpha \norm{\bm{x}-\bm{y}}^2 \\
        \langle \nabla f(\bm{x}) - \nabla f(\bm{y}) - \alpha(\bm{x} - \bm{y}), \bm{x} - \bm{y} \rangle &\leq 0 \\
        \langle \nabla g(\bm{y}) - \nabla g(\bm{x}), \bm{y}-\bm{x} \rangle &\geq 0
    \end{align*}
    Therefore, by the monotone gradient property of convex functions, \(g(\bm{x})\) is convex. Hence, \(f\) is concavifiable and the smallest concavifier is \(\alpha^* = L\).
\end{proof}

We also show another characterization of concavifiable functions with the aid of mid-point acceleration function, as defined below.
\paragraph{Definition 4: Mid-point acceleration function \cite{zlobec2006characterization}} For a continuous function \(f: \mathbb{R}^d \rightarrow \mathbb{R}\) and a compact set \(C\) in \(\mathbb{R}^d \), the mid-point acceleration function \(\Psi\) is given by 
\begin{equation}
    \Psi(\bm{x},\bm{y}) = \dfrac{4}{\norm{\bm{x}-\bm{y}}^2} \left[f(\bm{x}) + f(\bm{y}) - 2 f \left( \dfrac{\bm{x} + \bm{y}}{2} \right) \right], \bm{x,y} \in C
\end{equation}
\begin{lemma} \label{lemma:midpoint}
    A function is concavifiable with concavifier \(\alpha\) if and only if the mid-point acceleration function is bounded above by \(\alpha\).
\end{lemma}
\begin{proof}
    Using the property of convexity of \(g(\bm{x})\), we use the following inequality
    \begin{equation}
        g \left( \dfrac{\bm{x} + \bm{y}}{2} \right)  \leq \dfrac{1}{2} \left[ g(\bm{x}) + g(\bm{y})\right]
    \end{equation}
    Directly substituting for \(g(\bm{x})\) as \(\dfrac{\alpha}{2} \norm{\bm{x}}^2 - f(\bm{x})\) and rearranging, we get 
    \begin{equation}
        \Psi(\bm{x},\bm{y}) \leq \alpha
    \end{equation}
\end{proof}

In the next section, we discuss the application of the concept of concavifiability to neural networks and work towards deriving the optimal step size while training.

\section{Application to neural networks}
    We consider the problem of training a neural network where an objective function is minimized with respect to the weights of the network. Recently, there has been a lot of interest in neural networks that use Rectified Linear Unit (ReLU) as their activation function. It is shown to perform well in different applications, especially image processing \cite{he2015delving}. The function does not conform to the assumptions usually made while analysing the convergence of neural networks like differentiability and smooth gradients and hence has sparked a new line of research to provide theoretical convergence guarantees \cite{li2017convergence,soltanolkotabi2017learning,virmaux2018lipschitz}. 
    
    Consider a data set with feature vector \(\bm{x_i} \in \mathbb{R}^d\) and output \(y_i \in \mathbb{R}\) for \(i=1,...N\). We consider a single hidden layer neural network with \(k\) neurons with ReLU activation. Let us denote the weight vector from the input to the \(j\)th hidden layer neuron as \(\bm{w^j}\). Note that \(\bm{w^j} \in \mathbb{R}^d\) for \(j = 1,...k\). The column vector \(\bm{w}\) refers to the stack of vectors \(\bm{w^1},...\bm{w^k}\). Note that \(\bm{w} \in \mathbb{R}^{kd}\). The output of the network is taken as the sum of outputs from each of the hidden layer neurons and is given by \(f(\bm{x},\bm{w}) = \sum_{j=1}^k \max(0,\bm{x}^T \bm{w^j})\) for input \(\bm{x}\). 
    
    We assume that there is a true underlying network with the same architecture with weights \(\bm{w^*}\) which produces the output \(f(\bm{x},\bm{w^*})\) for a given input \(\bm{x}\) as done in \cite{li2017convergence,saad1996dynamics}. We denote the training data as a set of points \((\bm{x_i},y_i)\) where \(\bm{x_i} \sim \mathcal{N}(\bm{0},I_d)\)  and \(y_i = f(\bm{x_i},\bm{w^*}) \in \mathbb{R}\) for \(i = 1,...N\).  
    
    \subsection{Computing concavifier in case of a single data point}
    When a quadratic loss is employed for this architecture, the minimization objective for a single data point \((\bm{x},y)\) is given by
     \begin{equation} \label{eqn:lossRelu}
              l(\bm{w}) = \dfrac{1}{2}\left( \left(\sum_{j=1}^k \max(0,\bm{x}^T \bm{w^j}) \right) - y \right)^2.
     \end{equation}
     
     \begin{theorem}\label{theorem:optConv}
         The concavifier for the loss function provided in Equation \ref{eqn:lossRelu} is given by
         \begin{equation}
            \alpha^* = k \norm{\bm{x}}^2
        \end{equation}
     \end{theorem}
     \begin{proof}
     We check for the concavifiability of the function \(l(\bm{w})\). By Lemma \ref{lemma:doubleDiff}, as the function is doubly differentiable, the concavifier is \(\alpha^* = \max_{\bm{w}} \lambda_{max}(\nabla^2 l(\bm{w}) )\) where \(\lambda_{max}\) refers to the maximum eigen value of the Hessian matrix.
     \begin{align*}
         \nabla l(\bm{w}) &= \left( \sum_{j=1}^k \max(0,\bm{x}^T \bm{w^j}) - y \right) 
         \begin{bmatrix} \mathbb{I}_{\{\bm{x}^T\bm{w^1} \geq 0\}}  \bm{x}\\ \vdots \\ \mathbb{I}_{\{\bm{x}^T\bm{w^k} \geq 0\}}  \bm{x} \end{bmatrix} 
    \end{align*}
    \begin{align*}
         \nabla^2 l(\bm{w}) &=  \begin{bmatrix} \mathbb{I}_{\{\bm{x}^T\bm{w^1} \geq 0\}} \bm{x} \\ \vdots \\ \mathbb{I}_{\{\bm{x}^T\bm{w^k} \geq 0\}}\bm{x} \end{bmatrix}  \begin{bmatrix} \mathbb{I}_{\{\bm{x}^T\bm{w^1} \geq 0\}}  \bm{x}\\ \vdots \\ \mathbb{I}_{\{\bm{x}^T\bm{w^k} \geq 0\}} \bm{x} \end{bmatrix} ^T =  \bm{a(x,w}) \bm{a(x,w})^T
     \end{align*}
     where we denote the vectors 
     \begin{multicols}{2}
        \begin{equation} \label{eqn:avec}
            \bm{a(x,w}) =\begin{bmatrix} \mathbb{I}_{\{\bm{x}^T\bm{w^1} \geq 0\}}  \bm{x}\\ \vdots \\ \mathbb{I}_{\{\bm{x}^T\bm{w^k} \geq 0\}}  \bm{x} \end{bmatrix}
        \end{equation}\break
        \begin{equation} \label{eqn:abar}
            \bm{\bar{a}(x)} = \begin{bmatrix} \bm{x}  \\ \vdots \\ \bm{x} \end{bmatrix} 
        \end{equation}
    \end{multicols}

     Note that \(\bm{a}\) is a function of the data vector \(\bm{x}\) and weights \(\bm{w}\) whereas \(\bm{\bar{a}(x)}\) depends only on the data \(\bm{x}\). They are both of dimension \(kd \times 1\).
     \paragraph{Remark} Although the ReLU function given by \(\max(0,x)\) is non-differentiable at \(x=0\), the work in \cite{li2017convergence} states that if the input is assumed to be from the Gaussian distribution, the loss function becomes smooth, and the gradient is well defined everywhere. The gradient is given by \(\mathbb{I}_{\{x \geq 0\}}\) where \(\mathbb{I}\) is the indicator function. We consider the second derivative to be zero over the entire real line.\\
     The concavifier is given by
     \begin{equation} \label{eqn:concOnePt}
         \alpha^* = \max_{\bm{w}} \lambda_{max}(\nabla^2 l(\bm{w}) )= \max_{\bm{w}} \lambda_{\max} ( \bm{a(x,w}) \bm{a(x,w})^T)
     \end{equation}
     
     We note that \(\bm{a(x,w}) \bm{a(x,w})^T\) is a rank-1 matrix and therefore, its only non-zero eigen value is given by \(\bm{a(x,w})^T \bm{a(x,w})= \norm{\bm{a(x,w})}^2\), which is also the maximum eigen value. Substituting in Equation \ref{eqn:concOnePt},
     \begin{equation} \label{eqn:bound1}
         \alpha^* = \max_{\bm{w}} \norm{\bm{a(x,w})}^2
     \end{equation}
     The norm is maximized when all the entries of the vector are non-zero, i.e., when all the indicators correspond to 1. Therefore, the optimal concavifier is given by 
     \begin{align}
         \alpha^* = \norm{\bm{\bar{a}(x)}}^2= k \norm{\bm{x}}^2
     \end{align}
     This gives us the optimal concavifier for a single input.
     \end{proof}
     
     \subsection{Extension to multiple inputs} 
     The loss function for the case of \(N\) inputs when the data set \((\bm{x_i},y_i)\) for \(i = 1,...N\) is employed is given by 
     \begin{equation} \label{eqn:lossNinputsrelu}
              l(\bm{w}) = \dfrac{1}{2N} \sum_{i = 1}^N \left(  \left(\sum_{j=1}^k \max(0,\bm{x_i}^T \bm{w^j}) \right) - y_i \right)^2.
     \end{equation}
     With a slight overload of notation, we call loss corresponding to \(N\) data points as \(l(\bm{w})\) as well. 
     The result in Theorem \ref{theorem:optConv} can be extended to \(N\) inputs as
     \begin{equation} \label{eqn:alphaNinputsrelu}
         \alpha^* = \dfrac{1}{N} \max_{\bm{w}} \lambda_{\max}\left( \sum_{i=1}^N  \bm{a(x_i,w}) \bm{a(x_i,w})^T  \right) 
     \end{equation}
     \revtwo{This is the exact value of the optimal concavifier. The computation of the optimal concavifier involves a search over all possible values the weights can take and hence is computationally heavy.} \rev{Note that any constant \(\alpha > \alpha^*\) also acts as a concavifier. Therefore, we propose the following bounds on the optimal concavifier.}
     
     \subsubsection{Upper bound based on sum of maximum eigen values}
         \begin{align}
         \alpha^* &= \dfrac{1}{N} \max_{\bm{w}} \lambda_{\max}\left( \sum_{i=1}^N  \bm{a(x_i,w}) \bm{a(x_i,w})^T  \right) \\
         &\leq \dfrac{1}{N}  \max_{\bm{w}} \sum_{i=1}^N  \left( \lambda_{\max} \left(  \bm{a(x_i,w}) \bm{a(x_i,w})^T \right) \right) \label{eqn:maxLambda}\\
        \alpha^* & \leq \dfrac{k}{N} \sum_{i=1}^N \norm{\bm{x_i}}^2 \label{eqn:final}
     \end{align}
     \revtwo{Let us denote this bound as \(\alpha_1\). Equation \ref{eqn:maxLambda} is obtained by using the following property of eigen values \(\lambda_{\max} (\bm{P}+\bm{Q}) \leq \lambda_{\max}(\bm{P}) + \lambda_{\max}(\bm{Q}) \) which follows from the fact that the largest eigen value can be expressed as the spectral norm and norms follow the triangle inequality. The \(\max\) and sum operator are interchanged as we can maximize each term individually since they are positive and use the result from Theorem \ref{theorem:optConv} } to obtain Equation \ref{eqn:final}.
     
     \subsubsection{Can we do better?}
     \revtwo{We note that the Hessian matrix in this specific problem is structured. We have a sum outer products of the vector \(\bm{a(x_i,w})\) where the vector consists of \(\bm{x_i}\) multiplied by appropriate indicators. We wish to exploit the structure of the Hessian matrix to arrive at a better bound for the optimal concavifier.} Towards that end we propose the following lemma.
     \begin{lemma} \label{lemma:betterBound}
         For a vector \(\bm{a}(\bm{x_i},\bm{w})\) as defined in Equation \ref{eqn:avec}, the following relation holds
         \begin{equation}
             \lambda_{max}\left( \sum_{i=1}^N  \bm{\bar{a}(x_i)} \bm{\bar{a}(x_i)} ^T \right) \geq \lambda_{max}\left( \sum_{i=1}^N \bm{a(x_i,w}) \bm{a(x_i,w})^T  \right) 
         \end{equation}
     \end{lemma}
    \begin{proof}
         The Rayleigh quotient of a Hermitian matrix \(\bm{A}\) and a non-zero vector \(\bm{g}\) is given by \(\dfrac{ \bm{g}^T \bm{A} \bm{g}}{ \bm{g}^T \bm{g}}\) and reaches the maximum eigen value when the vector \(\bm{g}\) is the eigen vector corresponding to the maximum eigen value \cite{van2007new}.
         \begin{equation}
             \lambda_{max}(\bm{A}) = \max_{\bm{g}: \norm{\bm{g}}=1} \bm{g}^T \bm{A} \bm{g} ,
         \end{equation}
         Also, observe that for any other vector of unit norm \(\bm{h} \neq \bm{g}\),
         \begin{equation} \label{eqn:mat}
             \bm{g}^T \bm{A} \bm{g} > \bm{h}^T \bm{A} \bm{h}.
         \end{equation}
         In the following proof, we denote the principal eigen vectors of \(\left( \sum_{i=1}^N  \bm{\bar{a}(x_i)} \bm{\bar{a}(x_i)} ^T \right)\), \( \left(\bm{a(x_i,w}) \bm{a(x_i,w})^T  \right)\) and \(\left( \sum_{i=1}^N \bm{a(x_i,w}) \bm{a(x_i,w})^T  \right) \) as \(\bm{\bar{g}}, \bm{g_i}\) and \(\bm{\hat{g}}\) respectively. 
         
         \begin{align*}
             \lambda_{max}\left( \sum_{i=1}^N  \bm{\bar{a}(x_i)} \bm{\bar{a}(x_i)} ^T \right) &= \bm{\bar{g}}^T \left( \sum_{i=1}^N  \bm{\bar{a}(x_i)} \bm{\bar{a}(x_i)} ^T \right) \bm{\bar{g}} \\
             &= \sum_{i=1}^N \bm{\bar{g}}^T \left( \bm{\bar{a}(x_i)} \bm{\bar{a}(x_i)} ^T \right) \bm{\bar{g}} \\
             & \geq \sum_{i=1}^N \bm{g_i}^T \left( \bm{\bar{a}(x_i)} \bm{\bar{a}(x_i)} ^T \right) \bm{g_i}^T  && \text{(Using Equation \ref{eqn:mat})}
        \end{align*}
        Note that as \(\left(\bm{a(x_i,w}) \bm{a(x_i,w})^T  \right)\) is a rank-1 matrix, the principal eigen vector is given by \( \bm{g_i} = \bm{a(x_i,w})\). 
        \begin{align} \label{eqn:step2}
            \sum_{i=1}^N \bm{g_i}^T \left( \bm{\bar{a}(x_i)} \bm{\bar{a}(x_i)} ^T \right) \bm{g_i}^T 
             &=  \sum_{i=1}^N \bm{a(x_i,w})^T  \left( \bm{\bar{a}(x_i)} \bm{\bar{a}(x_i)} \right) \bm{a(x_i,w}) 
        \end{align}
        Let us consider each term in the summation.
        \begin{align*}
            &\bm{a(x_i,w})^T  \left( \bm{\bar{a}(x_i)} \bm{\bar{a}(x_i)} ^T \right) \bm{a(x_i,w}) \\
            &=  \bigg(\bm{a(x_i,w})^T \bm{\bar{a}(x_i)} \bigg) \bigg( \bm{\bar{a}(x_i})^T \bm{a(x_i,w}) \bigg) && \text{(Regrouping)}\\
            &= \bigg(\sum_{j=1}^k \mathbb{I}_{\{\bm{x_i^T w^j \geq0\}}} \bm{x_i^T x_i} \bigg)\bigg( \sum_{j=1}^k \mathbb{I}_{\{\bm{x_i^T w^j \geq0\}}} \bm{x_i^T x_i} \bigg) && \text{(Expanding each inner product)}\\
            &= \bigg(\sum_{j=1}^k \mathbb{I}^2_{\{\bm{x_i^T w^j \geq0\}}} \bm{x_i^T x_i} \bigg)\bigg( \sum_{j=1}^k \mathbb{I}^2_{\{\bm{x_i^T w^j \geq0\}}} \bm{x_i^T x_i} \bigg) && \text{(Since \(\mathbb{I}_{\mathcal{E}} = \mathbb{I}_{\mathcal{E}}^2\))} \\
            &= \bigg(\sum_{j=1}^k \mathbb{I}_{\{\bm{x_i^T w^j \geq0\}}} \bm{x_i^T} \mathbb{I}_{\{\bm{x_i^T w^j \geq0\}}} \bm{ x_i} \bigg)\bigg(\sum_{j=1}^k \mathbb{I}_{\{\bm{x_i^T w^j \geq0\}}} \bm{x_i^T} \mathbb{I}_{\{\bm{x_i^T w^j \geq0\}}} \bm{ x_i}  \bigg)\\
            &=  \bm{a(x_i,w})^T  \left( \bm{a(x_i,w}) \bm{a(x_i,w}) ^T \right) \bm{a(x_i,w})
        \end{align*}
        Using this result in Equation \ref{eqn:step2},
        \begin{align*}
            \sum_{i=1}^N \bm{a(x_i,w})^T  \left( \bm{\bar{a}(x_i)} \bm{\bar{a}(x_i)} ^T \right) \bm{a(x_i,w}) 
             &= \sum_{i=1}^N \bm{a(x_i,w})^T  \left( \bm{a(x_i,w}) \bm{a(x_i,w}) ^T \right) \bm{a(x_i,w}) \\
             & \geq \sum_{i=1}^N \bm{\hat{g}}^T  \left( \bm{a(x_i,w}) \bm{a(x_i,w}) ^T \right)  \bm{\hat{g}} && \text{(Using Equation \ref{eqn:mat})}\\
             &=  \bm{\hat{g}}^T  \left(\sum_{i=1}^N \bm{a(x_i,w}) \bm{a(x_i,w}) ^T \right)  \bm{\hat{g}} \\
             &= \lambda_{max}\left( \sum_{i=1}^N \bm{a(x_i,w}) \bm{a(x_i,w})^T  \right) 
         \end{align*}
     \end{proof}
     Now, using Lemma \ref{lemma:betterBound}, we can rewrite Equation \ref{eqn:alphaNinputsrelu} as 
     \begin{equation} \label{eqn:alpha2}
          \alpha^* \leq \dfrac{1}{N} \lambda_{max}\left( \sum_{i=1}^N \bm{\bar{a}(x_i)} \bm{\bar{a}(x_i)} ^T  \right) 
     \end{equation}
     Let \(\alpha_2 = \frac{1}{N} \lambda_{max}\left( \sum_{i=1}^N \bm{\bar{a}(x_i)} \bm{\bar{a}(x_i)} ^T  \right)\). We now have an upper bound on the concavifier independent of the maximization over the weights. However, to compute the bound \(\alpha_2\) on the concavifier, we need to compute the maximum eigen value of the matrix, which might be computationally intensive for huge dimensions. Therefore, we propose using the following bounds for the computation of the maximum eigen value for symmetric matrices \cite{deville2019optimizing}. We are interested in the matrix
     \begin{equation}\bm{M}  =  \dfrac{1}{N} \sum_{i=1}^N \bm{\bar{a}(x_i)} \bm{\bar{a}(x_i)} ^T .
     \end{equation}
     
     \paragraph{Bound based on Gershgorin’s Circles }
     \revtwo{A well-known result in bounding the spectrum of eigen value of a square matrix is the Gershgorin’s Circles theorem \cite{varga2009matrix}. 
      Applying the Gershgorin’s theorem for the upper bound on the largest eigen value on the matrix \(\bm{M}\), we obtain an upper bound on the concavifier which we denote as \(\alpha_3\); it is given by}
     \begin{equation} \label{eqn:alpha3}
         \alpha^* \leq \max_i \left(m_{ii} + R_i(\bm{M})\right)
     \end{equation}
      where \(R_i(\bm{M}) = \sum_{i \neq j} |m_{ij}|\). 
     \paragraph{Brauer’s Ovals of Cassini}
     \revtwo{This bound was proposed to optimize Gershgorin's bound for symmetric matrices on the spectrum of eigen values. This bound is guaranteed to be provide a bound which is not worse than the Gershgorin bound \cite{deville2019optimizing}. Note that our matrix of interest \(\bm{M}\) is also symmetric. Hence, applying Brauer's upper bound on the maximum eigen value to the matrix \(\bm{M}\), we get another upper for the concavifier as \(\alpha_4\), given by}
     \begin{equation} \label{eqn:alpha4}
         \alpha^* \leq \max_{i \neq j} \left(\dfrac{m_{ii} + m_{jj}}{2} + \sqrt{(m_{ii} - m_{jj})^2 + R_i(\bm{M}) R_j(\bm{M})}  \right) 
     \end{equation}
     where \(R_i(\bm{M}) = \sum_{i \neq j} |m_{ij}|\). We denote this bound as \(\alpha_4 \).
     
     \subsection{Simulation Results}
     We validate the derived bounds through the help of simulations. The gradient descent algorithm is employed to minimize the loss function in Equation \ref{eqn:lossNinputsrelu} with respect to the weight vector \(\bm{w}\). The parameters used are \(d=10, k=5,N=1000\). We also explore the performance for other parameters in the supplementary material. The weights of the underlying network \(w^*\) are drawn from a multivariate Gaussian distribution with zero mean and identity as its covariance. For performing an explicit search to find the optimal concavifier involves a brute force search over the entire space of \(\bm{w}\) followed by an eigen value computation for a \(kd \times kd\) matrix. Therefore, the bounds derived for the maximum eigen values prove very useful when a larger data set is employed.\\
     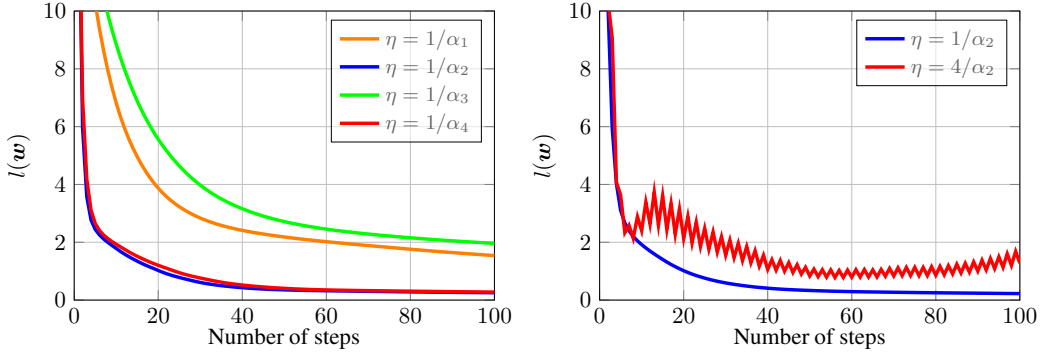
\begin{figure}[!t]
    \centering
    \begin{subfigure}{.5\textwidth}
        \resizebox{\linewidth}{!}{
\begin{tikzpicture}[thick]
    \begin{axis}[
        width=8cm,
        height=6cm,
        xmin=0,
        xmax=100,
        ymin=0.00,
        ymax=10.00,
        grid=major,
        xlabel={Number of steps},
        ylabel={\(l(\bm{w})\)},
        xlabel style={at={(0.50,0.05)}},
        ylabel style={at={(0.06,0.50)}},
        ytick={0.00,2.00,...,10.00},
        log ticks with fixed point,
        legend pos=north east,
        legend cell align={left},
        legend style={fill opacity=0.6, draw opacity=1.0, text opacity=1.0, font=\small}
        ]
        
        \addplot[orange, solid, thick, line width = 1.5pt]
            table [x=alpha1_x, y=alpha1_y, col sep=comma]{./data_bndFig.csv};
        \addlegendentry{\(\eta = 1/\alpha_1\)};
        
        \addplot[blue, solid, thick, line width = 1.5pt]
            table [x=alpha2_x, y=alpha2_y, col sep=comma]{./data_bndFig.csv};
        \addlegendentry{\(\eta = 1/\alpha_2\)};
        
        \addplot[green, solid, thick, line width = 1.5pt]
            table [x=alpha3_x, y=alpha3_y, col sep=comma]{./data_bndFig.csv};
        \addlegendentry{\(\eta = 1/\alpha_3\)};
        
        \addplot[red, solid, thick, line width = 1.5pt]
            table [x=alpha4_x, y=alpha4_y, col sep=comma]{./data_bndFig.csv};
        \addlegendentry{\(\eta = 1/\alpha_4\)};





    \end{axis}
\end{tikzpicture}}
        \caption{Convergence for different bounds on concavifier}
        \label{fig:Relu_bound}
    \end{subfigure}%
    \begin{subfigure}{.5\textwidth}
        \resizebox{\linewidth}{!}{
\begin{tikzpicture}[thick]
    \begin{axis}[
        width=8cm,
        height=6cm,
        xmin=0,
        xmax=100,
        ymin=0.00,
        ymax=10.00,
        grid=major,
        xlabel={Number of steps},
        ylabel={\(l(\bm{w})\)},
        xlabel style={at={(0.50,0.05)}},
        ylabel style={at={(0.06,0.50)}},
        ytick={0.00,2.00,...,10.00},
        log ticks with fixed point,
        legend pos=north east,
        legend cell align={left},
        legend style={fill opacity=0.6, draw opacity=1.0, text opacity=1.0, font=\small}
        ]
        
        
        
        \addplot[blue, solid, thick, line width = 1.5pt]
            table [x=a3_x, y=a3_y, col sep=comma]{./data_scales.csv};
        \addlegendentry{\(\eta = 1/ \alpha_2\)};
        
        
        \addplot[red, solid, thick,line width = 1.5pt]
            table [x=a5_x, y=a5_y, col sep=comma]{./data_scales.csv};
        \addlegendentry{\(\eta = 4/ \alpha_2\)};





    \end{axis}
\end{tikzpicture}}
        \caption{Convergence for scaled values of the best bound \(\alpha_2\)}
        \label{fig:Relu_scale}
    \end{subfigure}%
    \caption{Convergence of weights: Variation of \(l(\bm{w})\) with number of steps for \(d=10, k=5,N=1000\)}
    \label{fig:wevo_curve}
\end{figure}
     The convergence curves for values of concavifier obtained through the different bounds that are derived in the Equations \ref{eqn:final}, \ref{eqn:alpha2}, \ref{eqn:alpha3} and \ref{eqn:alpha4} are plotted in Figure \ref{fig:Relu_bound}. Note that for concavifier \(\alpha\), the optimal step size is \(1/\alpha\). Here, we reiterate that we provide upper bounds to \(\alpha^*\) and hence, convergence is guaranteed. To verify the tightness of our bound, we pick the best bound on the concavifier observed from Figure \ref{fig:Relu_bound}, i.e., \(\alpha_2\); we then scale the step sizes to observe the convergence curves as plotted in Figure \ref{fig:Relu_scale}. It is noted that for step sizes in the order of \(4/ \alpha_2\), we do not observe a strictly decreasing sequence. Therefore, we can see that our derived bound will be very helpful to select the learning rate while training.

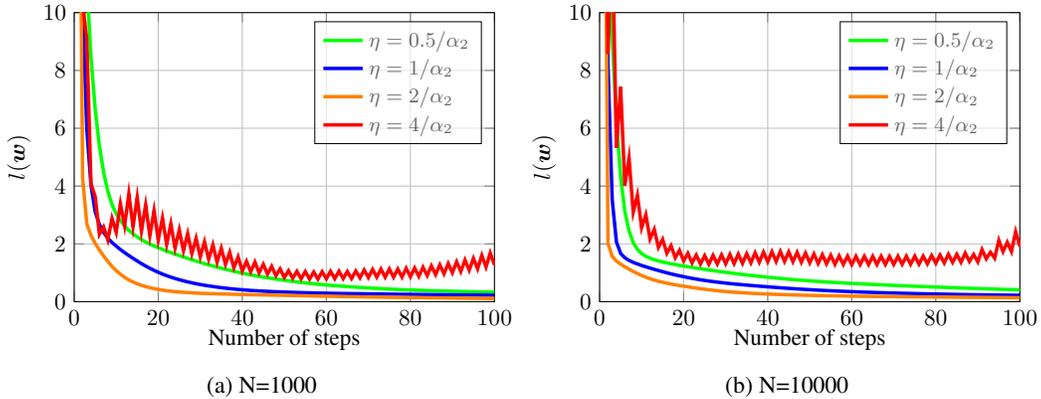
\begin{figure}[!h]
    \centering
    \begin{subfigure}{.5\textwidth}
        \resizebox{\linewidth}{!}{
\begin{tikzpicture}[thick]
    \begin{axis}[
        width=8cm,
        height=6cm,
        xmin=0,
        xmax=100,
        ymin=0.00,
        ymax=10.00,
        grid=major,
        xlabel={Number of steps},
        ylabel={\(l(\bm{w})\)},
        xlabel style={at={(0.50,0.05)}},
        ylabel style={at={(0.06,0.50)}},
        ytick={0.00,2.00,...,10.00},
        log ticks with fixed point,
        legend pos=north east,
        legend cell align={left},
        legend style={fill opacity=0.6, draw opacity=1.0, text opacity=1.0, font=\small}
        ]
        
        
        \addplot[green, solid, thick, line width = 1.5pt]
            table [x=a2_x, y=a2_y, col sep=comma]{./data_scales.csv};
        \addlegendentry{\(\eta = 0.5/ \alpha_2\)};
        
        \addplot[blue, solid, thick, line width = 1.5pt]
            table [x=a3_x, y=a3_y, col sep=comma]{./data_scales.csv};
        \addlegendentry{\(\eta = 1/ \alpha_2\)};
        
        \addplot[orange, solid, thick, line width = 1.5pt]
            table [x=a4_x, y=a4_y, col sep=comma]{./data_scales.csv};
        \addlegendentry{\(\eta = 2/ \alpha_2\)};
        
        \addplot[red, solid, thick,line width = 1.5pt]
            table [x=a5_x, y=a5_y, col sep=comma]{./data_scales.csv};
        \addlegendentry{\(\eta = 4/ \alpha_2\)};





    \end{axis}
\end{tikzpicture}}
        \caption{N=1000}
    \end{subfigure}%
    \begin{subfigure}{.5\textwidth}
        \resizebox{\linewidth}{!}{
\begin{tikzpicture}[thick]
    \begin{axis}[
        width=8cm,
        height=6cm,
        xmin=0,
        xmax=100,
        ymin=0.00,
        ymax=10.00,
        grid=major,
        xlabel={Number of steps},
        ylabel={\(l(\bm{w})\)},
        xlabel style={at={(0.50,0.05)}},
        ylabel style={at={(0.06,0.50)}},
        ytick={0.00,2.00,...,10.00},
        log ticks with fixed point,
        legend pos=north east,
        legend cell align={left},
        legend style={fill opacity=0.6, draw opacity=1.0, text opacity=1.0, font=\small}
        ]
        
        
        \addplot[green, solid, thick, line width = 1.5pt]
            table [x=a2_x, y=a2_y, col sep=comma]{./sup_datScaled10k5N10000.csv};
        \addlegendentry{\(\eta = 0.5/ \alpha_2\)};
        
        \addplot[blue, solid, thick, line width = 1.5pt]
            table [x=a3_x, y=a3_y, col sep=comma]{./sup_datScaled10k5N10000.csv};
        \addlegendentry{\(\eta = 1/ \alpha_2\)};
        
        \addplot[orange, solid, thick, line width = 1.5pt]
            table [x=a4_x, y=a4_y, col sep=comma]{./sup_datScaled10k5N10000.csv};
        \addlegendentry{\(\eta = 2/ \alpha_2\)};
        
        \addplot[red, solid, thick,line width = 1.5pt]
            table [x=a5_x, y=a5_y, col sep=comma]{./sup_datScaled10k5N10000.csv};
        \addlegendentry{\(\eta = 4/ \alpha_2\)};





    \end{axis}
\end{tikzpicture}}
        \caption{N=10000}
    \end{subfigure}
    \caption{Effect of varying number of data points on the convergence by scaling \(\alpha_2\)for \(d =10\) and \(k =5\)}
    \label{fig:N}
\end{figure}
To study the variation of the performance with respect to change in the number of data points, we compare between \(N=1000\) and \(N= 10000\) points  as shown in Figure \ref{fig:N}. We conclude that the bound on the concavifier can act as a guide while performing hyper-parameter tuning instead of searching over a larger range.\\
     
      We now study the variation in the bounds for different values of the input dimension, \(d\), the number of neurons, \(k\) and the number of data points \(N\). The values for bounds obtained are tabulated in Table \ref{tab:bounds}.
\begin{table}[!h]
    \centering
    \begin{tabular}{|c|c|c|c|c|c|c|}
    \hline
     & \multicolumn{2}{|c|}{\(d=10,k=5\)} & \multicolumn{2}{|c|}{\(N = 1000, k = 5\)} &
     \multicolumn{2}{|c|}{\(d=10,N=1000\)} \tabularnewline
     \hline
        Bound &  N=1000 & N=10000 & d=5 & d= 50 & k=2 & k=50\\
         \hline
      \(\alpha_1\) & 48.8230 & 50.1020 & 24.4085 &249.4196 & 10.0366 &250.6084\\
      \(\alpha_2\) & 5.8450 & 5.2639 & 5.2798&7.0231 & 2.2672 & 53.09 \\
      \(\alpha_3\) & 76.2652& 78.4085 &33.5060 &503.4564 & 13.7343&341.6694\\
      \(\alpha_4\) & 6.9137 & 5.4655 &5.5278 & 12.7285& 2.4322&54.92\\
      \hline
    \end{tabular}
    \caption{Bounds obtained for different parameters}
    \label{tab:bounds}
\end{table}

From the Table \ref{tab:bounds}, it is evident that over different configurations, \(\alpha_2\) produces the lowest upper bound on the concavifier.Observe that \(\alpha_4\) is a close approximation of \(\alpha_2\) even at higher dimensions whereas the other bounds provide grossly overestimated bounds of the concavifier. Therefore, for higher dimensions, when the explicit computation of the maximum eigen value becomes computationally intensive, the bound \(\alpha_4\) can be employed with hardly any degradation in the convergence performance.
   
\section{Conclusions}
In this paper, we questioned the precondition that we require the objective function to be gradient Lipschitz continuous for analyzing the convergence of the gradient descent algorithm. We introduced a class of functions known as concavifiable functions and showed that concavifiability is necessary and sufficient to show that the objective function decreases in value on the application of the gradient descent algorithm. This implies that all results that show the decrease of objective for gradient Lipschitz function can be extended for concavifiable functions as well. We also explicitly devise a way to derive the concavifier; this enables us to give a bound on the step size such that gradient descent will always converge to the first order stationary point. \rev{We derive the expression for the concavifier for a single hidden layer neural network with ReLU activation functions. As the explicit computation is hard, we provide bounds on the concavifier which can act as a guide for a more refined hyper-parameter search. This is a preliminary step to compute the concavifier for a shallow neural network analytically. A possible future direction would be to extend the analysis to a network of multiple layers and attempting to tighten the bound for the concavifier.}

\bibliographystyle{plain}

\medskip

\small
\begin{thebibliography}{24}

\bibitem{battiti1992first}
Roberto Battiti.
\newblock First-and second-order methods for learning: between steepest descent
  and newton's method.
\newblock {\em Neural computation}, 4(2):141--166, 1992.

\bibitem{bergstra2012random}
James Bergstra and Yoshua Bengio.
\newblock Random search for hyper-parameter optimization.
\newblock {\em Journal of Machine Learning Research}, 13(Feb):281--305, 2012.

\bibitem{bolte2018first}
J{\'e}r{\^o}me Bolte, Shoham Sabach, Marc Teboulle, and Yakov Vaisbourd.
\newblock First order methods beyond convexity and lipschitz gradient
  continuity with applications to quadratic inverse problems.
\newblock {\em SIAM Journal on Optimization}, 28(3):2131--2151, 2018.

\bibitem{bottou2010large}
L{\'e}on Bottou.
\newblock Large-scale machine learning with stochastic gradient descent.
\newblock In {\em Proceedings of COMPSTAT'2010}, pages 177--186. Springer,
  2010.

\bibitem{deville2019optimizing}
Lee DeVille.
\newblock Optimizing gershgorin for symmetric matrices.
\newblock {\em Linear Algebra and its Applications}, 2019.

\bibitem{du2017gradient}
Simon~S Du, Chi Jin, Jason~D Lee, Michael~I Jordan, Aarti Singh, and Barnabas
  Poczos.
\newblock Gradient descent can take exponential time to escape saddle points.
\newblock In {\em Advances in neural information processing systems}, pages
  1067--1077, 2017.

\bibitem{he2015delving}
Kaiming He, Xiangyu Zhang, Shaoqing Ren, and Jian Sun.
\newblock Delving deep into rectifiers: Surpassing human-level performance on
  imagenet classification.
\newblock In {\em Proceedings of the IEEE international conference on computer
  vision}, pages 1026--1034, 2015.

\bibitem{jin2017escape}
Chi Jin, Rong Ge, Praneeth Netrapalli, Sham~M Kakade, and Michael~I Jordan.
\newblock How to escape saddle points efficiently.
\newblock In {\em Proceedings of the 34th International Conference on Machine
  Learning-Volume 70}, pages 1724--1732. JMLR. org, 2017.

\bibitem{johnson2013accelerating}
Rie Johnson and Tong Zhang.
\newblock Accelerating stochastic gradient descent using predictive variance
  reduction.
\newblock In {\em Advances in neural information processing systems}, pages
  315--323, 2013.

\bibitem{khamaru2018convergence}
Koulik Khamaru and Martin Wainwright.
\newblock Convergence guarantees for a class of non-convex and non-smooth
  optimization problems.
\newblock In {\em International Conference on Machine Learning}, pages
  2606--2615, 2018.

\bibitem{Lee2019}
Jason~D. Lee, Ioannis Panageas, Georgios Piliouras, Max Simchowitz, Michael~I.
  Jordan, and Benjamin Recht.
\newblock First-order methods almost always avoid strict saddle points.
\newblock {\em Mathematical Programming}, Feb 2019.

\bibitem{li2017convergence}
Yuanzhi Li and Yang Yuan.
\newblock Convergence analysis of two-layer neural networks with relu
  activation.
\newblock In {\em Advances in Neural Information Processing Systems}, pages
  597--607, 2017.

\bibitem{nesterov1998introductory}
Yurii Nesterov.
\newblock Introductory lectures on convex programming volume i: Basic course.
\newblock {\em Lecture notes}, 1998.

\bibitem{nitanda2014stochastic}
Atsushi Nitanda.
\newblock Stochastic proximal gradient descent with acceleration techniques.
\newblock In {\em Advances in Neural Information Processing Systems}, pages
  1574--1582, 2014.

\bibitem{saad1996dynamics}
David Saad and Sara~A Solla.
\newblock Dynamics of on-line gradient descent learning for multilayer neural
  networks.
\newblock In {\em Advances in neural information processing systems}, pages
  302--308, 1996.

\bibitem{shamir2013stochastic}
Ohad Shamir and Tong Zhang.
\newblock Stochastic gradient descent for non-smooth optimization: Convergence
  results and optimal averaging schemes.
\newblock In {\em International Conference on Machine Learning}, pages 71--79,
  2013.

\bibitem{snoek2012practical}
Jasper Snoek, Hugo Larochelle, and Ryan~P Adams.
\newblock Practical bayesian optimization of machine learning algorithms.
\newblock In {\em Advances in neural information processing systems}, pages
  2951--2959, 2012.

\bibitem{soltanolkotabi2017learning}
Mahdi Soltanolkotabi.
\newblock Learning relus via gradient descent.
\newblock In {\em Advances in Neural Information Processing Systems}, pages
  2007--2017, 2017.

\bibitem{van2007new}
Piet Van~Mieghem.
\newblock A new type of lower bound for the largest eigenvalue of a symmetric
  matrix.
\newblock {\em Linear Algebra and its Applications}, 427(1):119--129, 2007.

\bibitem{varga2009matrix}
Richard~S Varga.
\newblock {\em Matrix iterative analysis}, volume~27.
\newblock Springer Science \& Business Media, 2009.

\bibitem{virmaux2018lipschitz}
Aladin Virmaux and Kevin Scaman.
\newblock Lipschitz regularity of deep neural networks: analysis and efficient
  estimation.
\newblock In {\em Advances in Neural Information Processing Systems}, pages
  3835--3844, 2018.

\bibitem{zhou2018fenchel}
Xingyu Zhou.
\newblock On the fenchel duality between strong convexity and lipschitz
  continuous gradient.
\newblock {\em arXiv preprint arXiv:1803.06573}, 2018.

\bibitem{zlobec2003estimating}
Sanjo Zlobec.
\newblock Estimating convexifiers in continuous optimization.
\newblock {\em Mathematical Communications}, 8(2):129--137, 2003.

\bibitem{zlobec2006characterization}
Sanjo Zlobec.
\newblock Characterization of convexifiable functions.
\newblock {\em Optimization}, 55(3):251--261, 2006.

\end{thebibliography}

\end{document}